\documentclass[a4paper]{article}
\usepackage{lmodern}
\usepackage[mathjax]{lwarp}
\usepackage{stmaryrd,amsmath,amssymb,ifthen,mathrsfs,hyperref,youngtab,graphicx}
\usepackage{amsthm}

\newtheorem{prop}{Proposition}
\newtheorem{thm}[prop]{Theorem}

\theoremstyle{definition}

\newtheorem{conj}[prop]{Conjecture}
\newtheorem{ques}[prop]{Question}
\theoremstyle{remark}

\newcommand{\mathmacro}[1]{#1\CustomizeMathJax{#1}}

\mathmacro{

\newcommand{\Q}{\mathbb{Q}}

\newcommand{\GL}{\operatorname{GL}}

\newcommand{\PSL}{\operatorname{PSL}}

\newcommand{\Irr}{\operatorname{Irr}}

\newcommand{\cd}{\mathrm{cd}}
\newcommand{\cc}{\mathrm{cc}}

}

\numberwithin{prop}{section}
\numberwithin{equation}{section}
\numberwithin{table}{section}
\numberwithin{figure}{section}

\title{On the irreducible character degrees of symmetric groups and their multiplicities}
\author{David A.\ Craven\\School of Mathematics, University of Birmingham,\\Edgbaston, Birmingham, B15 2TT, UK\\\texttt{d.a.craven@bham.ac.uk}}
\date{\today}

\begin{document}
\maketitle

\begin{abstract} We consider problems concerning the largest degrees of irreducible characters of symmetric groups, and the multiplicities of character degrees of symmetric groups. Using evidence from computer experiments, we posit several new conjectures or extensions of previous conjectures, and prove a number of results. One of these is that, if $n\geq 21$, then there are at least eight irreducible characters of $S_n$, all of which have the same degree, and which have irreducible restriction to $A_n$. We explore similar questions about unipotent degrees of $\GL_n(q)$. We also make some remarks about how the experiments here shed light on posited algorithms for finding the largest irreducible character degree of $S_n$.

\end{abstract}

\section{Introduction}

The multisets of degrees of the irreducible characters of the symmetric groups $S_n$ have been studied for over a century. Over the last 50 years or so there has been work on two particular aspects of this: the largest degrees of irreducible characters, and the multiplicities of irreducible character degrees. Here we will propose a few new conjectures and avenues of research, on the basis of experiments with symmetric groups of degree up to either 129 or 150 (depending on the question), along with some new results in this area.

We start with multiplicities. If $G$ is a finite group, write $m(G)$ for the maximal multiplicity of an irreducible character degree of $G$, and write $m(n)=m(S_n)$. In \cite{craven2008} it was proved that $m(n)\to \infty$ as $n\to\infty$, and together with \cite{moreto2007} and the Classification of Finite Simple Groups, this proves that $m(G)\to\infty$ as $|G|\to\infty$, i.e., $|G|$ is bounded above by a function of $m(G)$. A natural question is to ask for the growth rate of $m(G)$ in general, and $m(n)$ in particular. For groups of Lie type $G(q)$ there are fairly easy lower bounds, for example roughly $q/4$ for $G(q)=\PSL_2(q)$.

An explicit lower bound for the growth of $m(n)$, approximately $n^{1/7}$, was proved in \cite[Theorem E]{craven2010a}, but this is far from optimal, and the precise nature of this bound is a consequence of the methods of \cite{craven2008}. The theorem we prove here is the following, which has been used in \cite{martinezmadrid2024un} to find all finite groups $G$ for which $m(G)=2$ (the Baby Monster is the largest such group). This result is a slight specialization of Theorem \ref{thm:boundsfordegrees} below.

\begin{thm} If $n\geq 11$ then $m(n)\geq 4$. If $n\geq 17$ then $m(n)\geq 6$. If $n\geq 21$ then $m(n)\geq 8$. In each of these cases, it is possible to choose four, six and eight irreducible characters respectively that are of the same degree and also have irreducible restriction to $A_n$.
\end{thm}

For small values of $n$ it is possible to prove this theorem by simply finding a number of characters of the same degree. Above this we need a blend of techniques from \cite{craven2008} and computer methods to complete the proof. Using a computer, in Table \ref{t:maxmult} we give the value of $m(n)$ for $n\leq 129$. For $n\geq 206$ one can prove the result by explicitly giving eight characters with the same degree. Between these bounds one needs to prove the result using a computer; this yields explicit characters but there is no generic nature to them, unlike the situation for $n\geq 206$.

The irreducible characters of symmetric groups are described combinatorially, see Section \ref{sec:sn} below for more information and background. The unipotent characters of the general linear groups, a specific, important subset of the irreducible characters of $\GL_n(q)$ that are defined independently of $q$, are parametrized by partitions, as for the irreducible characters of symmetric groups. Their degrees are combinatorially defined in a similar manner to the hook formula for symmetric groups (see Section \ref{sec:gln} below or \cite[p.465]{carterfinite} for an alternative version). It is reasonable to suspect that one might be able to say something about multiplicities of character degrees for general linear groups. As we see in Section \ref{sec:gln}, the degrees of the characters associated to a partition and its conjugate, while equal for $S_n$, are generally not for $\GL_n(q)$, so it is not even clear that there are two unipotent characters with the same degree for $\GL_n(q)$. To the author's knowledge there is no established theory about this in the literature, so this paper starts the investigation of this subject, with the following contribution.

\setcounter{section}{3}
\setcounter{prop}{6}

\begin{prop} If $\lambda$ is a partition of $n$, let $\chi_\lambda$ denote the corresponding unipotent character of $\GL_n(q)$. If $n=15,16,19,20,21,22$ or $n\geq 24$ then there is a partition $\lambda$ such that the generic degrees of $\lambda$ and the conjugate partition $\lambda'$ are the same. Otherwise there is no such partition.
\end{prop}

As $n$ increases the number of partitions $\lambda$ with this property seems to grow quickly (see Table \ref{t:glnpartitionssamedegree}), but the author can envisage no obvious proof of this.

\medskip

Connected to the maximal multiplicity of a character degree we have the average multiplicity of a character degree. Write $k(G)$ for the number of conjugacy classes of a finite group $G$, or equivalently the number of irreducible characters of $G$. Let $\cd(G)$ denote the set of irreducible character degrees of $G$, and $\cc(G)$ denote the set of conjugacy class sizes of $G$, the `dual notion'. The average character degree multiplicity is $k(G)/|\cd(G)|$, and the average conjugacy class size multiplicity is $k(G)/|\cc(G)|$. We compute both of these for symmetric groups $S_n$ up to $n=129$; a pattern emerges for $k(G)/|\cd(G)|$ (see Figure \ref{fig:avmult}) but the asymptotics remain unclear at this stage. For conjugacy classes, however, data up to $S_{85}$ is already enough for a convincing conjecture to emerge (see Figure \ref{fig:avmultcc}). The following two conjectures are considered in more detail, with current progress on them, in Section \ref{sec:mults}, and the concepts have been considered before, for example in \cite{jaikin2005,moreto2007}.

\setcounter{prop}{2}

\begin{conj} There exists a function $f:\Q\to\Q$ such that, if $G$ is a finite group with no alternating composition factor, then $|G|\leq f(k(G)/|\cd(G)|)$.
\end{conj}

\begin{conj} There exists a function $f:\Q\to\Q$ such that, if $G$ is a finite group then $|G|\leq f(k(G)/|\cc(G)|)$.
\end{conj}

\medskip

We also consider the largest degrees of $S_n$. Let $b_1(n)$ denote the largest character degree of $S_n$, $b_2(n)$ the second largest, $b_3(n)$ the third largest and so on. In \cite{hhn2016}, it was proved that, for $n\geq 7$, we have that
\[ \sum_{\substack{\chi\in\Irr(S_n)\\\chi(1)<b_1(n)}} \chi(1)^2>2b_1(n)^2.\]

Experimental data about the submaximal degrees $b_2(n)$, $b_3(n)$, and so on, collated in Section \ref{sec:chardegs}, gives a solid basis for believing the following conjecture.
\setcounter{section}{4}
\setcounter{prop}{0}
\begin{conj}  Let $n$ be an integer with $n\geq 5$, $n\neq 6,7$.
\begin{enumerate}
\item If $n\neq 11$ then $b_1(n)^2\leq b_2(n)^2+b_3(n)^2$.
\item If $n\neq 6,7,11,16,38$, and $b_2(n)$ or $b_3(n)$ is achieved by two or more irreducible characters of $S_n$, then
\[ \sum_{\substack{\chi\in\Irr(S_n)\\\chi(1)\in \{b_2(n),b_3(n)\}}}\chi(1)^2>2b_1(n)^2.\]
\item If $b_2(n)$ and $b_3(n)$ are both achieved by a unique irreducible character of $S_n$ then
\[ \sum_{\substack{\chi\in\Irr(S_n)\\\chi(1)\in \{b_2(n),b_3(n),b_4(n)\}}}\chi(1)^2>2b_1(n)^2.\]
\item If $n\geq 62$ then $2b_1(n)^2\leq b_2(n)^2+b_3(n)^2+b_4(n)^2$.
\end{enumerate}
\end{conj}

It appears as though the ratio $b_2(n)/b_1(n)$ tends to $1$ as $n$ grows, so that there is no single large degree for symmetric groups, see Figure \ref{fig:b3b1}. We posit that, in fact, this is true for any fixed $m$, so $b_m(n)/b_1(n)\to 1$ as $n\to\infty$.

\begin{conj} For any $m\geq 2$, the ratio $b_m(n)/b_1(n)\to 1$ as $n\to\infty$.
\end{conj}
\setcounter{section}{1}

This should be a consequence of there being many `large-degree' characters clustered around the maximal degree, but as far as the author is aware this result is not known.

At the end of Section \ref{sec:chardegs} we also include some information about how far the partitions $\lambda^{(n)}$ witnessing $b_1(n)$ stray from being self-conjugate. Contrary to what might be considered at first glance reasonable -- that as $n$ increases the partition becomes closer and closer to self-conjugate (mimicking the curve known to maximize the continuous version of the hook length formula found in \cite{loganshepp1977,vershikkerov1977}) -- it seems that the difference between $\lambda^{(n)}$ and a self-conjugate partition is unbounded as $n$ grows. This has implications for the posited algorithms for guessing the largest-degree irreducible character that start by looking at self-conjugate partitions, or the largest-degree irreducible character for $S_{n-1}$.

\medskip

All the conjectures and results herein were based on calculations using Magma \cite{magma}, and in particular the package \texttt{symmetric}, developed by the author. This package contains programs to compute all information described in this article, and most information is furthermore stored in databases in the package, for ease of access. The package, together with documentation on the included programs, is available on the author's website, GitHub, and the arXiv page for this article.

I would like to thank Benjamin Sambale for noting an error in the formulation of Conjectures \ref{conj:avmult} and \ref{conj:avmultcc}. I would like to thank the two referees for their helpful and incisive comments, that improved a number of facets of this work.

\section{Preliminaries}

Our notation is almost entirely standard.

\subsection{Symmetric groups}
\label{sec:sn}
Here we summarize the theory of the symmetric groups that we need, together with some terms and results from \cite{craven2008} that will be useful. A general introduction to the theory of representations of symmetric groups can be found in a number of places, for example \cite{james,jameskerber,sagan}.

The irreducible representations/characters of the symmetric group $S_n$ over a field of characteristic $0$ are labelled by partitions $\lambda$. The most important part of this for us is the \emph{hook formula}. First, we note the pictorial description of partitions known as \emph{Young diagrams} (also known as Ferrers diagrams). Ordering the parts of $n$ in the partition $\lambda=(\lambda_1,\dots,\lambda_r)$ so that $\lambda_i\geq \lambda_{i+1}$ for all $i$ (for us, partitions satisfy $\lambda_i>0$ for all parts), the Young diagram has $\lambda_i$ boxes in row $i$, aligned left. For example, the Young diagram for the partitions $(5,4,1)$ is below.
\[ \yng(5,4,1)\]
The \emph{conjugate partition} is the partition obtained by reflecting the Young diagram in the diagonal line in the $y=-x$ direction through the top-left box. For example, the conjugate partition of $(5,4,1)$ is $(3,2,2,2,1)$.
\[ \yng(5,4,1)\qquad \yng(3,2,2,2,1)\]
Given a box in the Young diagram of a partition $\lambda$, the \emph{hook} consists of the box itself, all boxes below it in $\lambda$, and all boxes to the right of it. The \emph{length} of the hook is the number of boxes in the hook. For example, the lengths of the hook in the boxes of $(5,4,1)$ are given below.
\[ \young(75431,5321,1)\]
Notice that a partition and its conjugate have the same multiset of hook lengths. We denote this multiset of hook lengths by $H(\lambda)$, and wish to emphasise that it is a \textbf{multiset}, not a set.

If $\chi_\lambda$ denotes the irreducible character of $S_n$ labelled by $\lambda$ then the degree of $\chi_\lambda$ is
\[ \chi_\lambda(1)=\frac{n!}{\prod_{h\in H(\lambda)} h},\]
where the product runs over hook lengths in $\lambda$.

At one point we will use the branching rule. If $\lambda$ is a partition of $n$, then a \emph{removable box} is a box in the Young diagram of $\lambda$ that, upon removal, yields a Young diagram of a (smaller) partition. This is equivalent to the box having hook length $1$. The opposite is an addable box: an \emph{addable box} is a box that is \emph{not} in the Young diagram of $\lambda$ but, if it is added to $\lambda$, the result is also the Young diagram of a partition.

The branching rule states that if $\lambda$ is a partition of $n$, then the restriction of $\chi_\lambda$ to $S_{n-1}$ is the sum of all irreducible characters $\chi_\mu$ such that $\mu$ is obtained from $\lambda$ by removing a removable box. By Frobenius reciprocity, the induction of $\chi_\lambda$ to $S_{n+1}$ is the sum of $\chi_\mu$ for $\mu$ obtained from $\lambda$ by adding an addable box.

\medskip

We now recap some definitions and results from \cite{craven2008}. Let $\Lambda$ be a set of partitions, all of the same integer $n$. The set $\Lambda$ is a \emph{cluster} if all elements of $\Lambda$ have the same hook lengths, i.e., $H(\lambda)=H(\mu)$ for all $\lambda,\mu\in\Lambda$. Since there is both the cardinality of the cluster and the number that the elements of the cluster are partitions of, we will say that the \emph{size} of the cluster is the size of the partitions in it, and the \emph{order} of the cluster is the number of elements in the cluster. A cluster is \emph{periodic} of period $p$ if, for all $n\geq 0$, adding $n$ boxes to the first $p$ rows of all partitions in $\Lambda$ forms another cluster. For example, the cluster $\{(5,2,2,2),(4,4,1,1,1)\}$ is periodic of period $3$, because $\{(n+5,n+2,n+2,2),(n+4,n+4,n+1,1,1)\}$ is also a cluster for all $n\geq 0$. (This example comes from \cite{hermanchung1978}.) If $n=0$ then the two partitions in the cluster are conjugate, but this is not the case for $n\geq 1$, so together with their conjugates we find four partitions with the same hook lengths for all integers of the form $3n+11$, $n\geq 1$.

To prove that a cluster is periodic of period $p$ we may use \cite[Theorem 4.2]{craven2008}, which asserts that a cluster is periodic if it satisfies two easy-to-check criteria. The first of these is that if one removes the first $p$ parts of all partitions in the cluster, we obtain another cluster. The second is that the multisets $\{(\lambda_i-\lambda_j)-(i-j)\mid 1\leq i<j\leq p\}$ are the same for all partitions $\lambda$ in the cluster.

\subsection{General linear groups}
\label{sec:gln}
The complex character theory of the general linear groups $\GL_n(q)$ uses many of the combinatorial concepts involved in symmetric groups such an hook lengths. One may find a description of some of this theory in \cite{carterfinite} and \cite{hissgroups}.

We consider a subset of the ordinary characters, called \emph{unipotent characters}. Although the general definition for arbitrary groups of Lie type is complicated, for $\GL_n(q)$ there is an easier description of these: they are the constituents of the permutation character on the Borel subgroup of upper triangular matrices.

These characters can be labelled by partitions of $n$, and their degrees are given by polynomials in $q$, called the \emph{generic degree} of the character (or of the corresponding partition). If $\lambda=(\lambda_1,\dots,\lambda_r)$ is a partition of $n$, define Lusztig's $a$-function by
\[a(\lambda)=\sum_{i=1}^r (i-1)\lambda_i.\]
Lusztig's $a$-function can be described as the smallest possible sum of entries in a semistandard tableau of shape $\lambda$ (with entries from $\{0,1,\dots\}$).
The degree of the character $\chi_\lambda$ is (see \cite[3.2.6]{hissgroups})
\[ \chi_\lambda(1)=q^{a(\lambda)}\frac{\prod_{i=1}^n (q^i-1)}{\prod_{h\in H(\lambda)} (q^h-1)}.\]
Thus $\chi_\lambda(1)=\chi_\mu(1)$ (for all $q$) if and only if $H(\lambda)=H(\mu)$ and $a(\lambda)=a(\mu)$.

Although two unipotent characters for $\GL_n(q)$ for a fixed $q$ could coincidentally have the same character degree, we are more interested in those partitions $\lambda$ and $\mu$ for which the generic degrees coincide, not just their value at a particular $q$. Notice that $\chi_\lambda(1)=\chi_\mu(1)$ if and only if $a(\lambda)=a(\mu)$ and the hook length multisets $H(\lambda)$ and $H(\mu)$ coincide. Thus this is a strictly stronger question than whether the symmetric group character degrees coincide. In general very little is known about the coincidence of generic degrees of unipotent characters.\footnote{If you are wondering how this reconciles with the statement in the Introduction that it was relatively easy to find many characters of groups of Lie type of the same degree, the characters with the same degree found in \cite{moreto2007} are not unipotent.}

In geometric and combinatorial versions of Brou\'e's conjecture, there is a bijection between the irreducible unipotent characters in a given block $B$ of a group of Lie type and those of its Brauer correspondent $b$. This bijection can be found using a deformation of the algebra $b$, called a cyclotomic Hecke algebra. The parameters of the cyclotomic Hecke algebra are known, and to each character of the cyclotomic Hecke algebra there is a generic degree attached to it, calculated from the parameters, which is the generic degree of one of the unipotent characters in the block $B$. If all of the generic degrees of the characters in $B$ are distinct, then there is a canonical bijection between the characters of $B$ and of $b$, going via this cyclotomic Hecke algebra. Proposition \ref{prop:alwaystwogln}, and particularly the data in Table \ref{t:glnpartitionssamedegree}, show that one cannot easily read off this bijection, and one actually has to calculate it, at least for some of the characters.\footnote{To check if two unipotent characters lie in the same block module $\ell$, we first find the multiplicative order $d$ of $q$ modulo $\ell$, and then $\chi_\lambda$ and $\chi_\mu$ lie in the same $\ell$-block if and only if $\lambda$ and $\mu$ have the same $d$-core. If $d=1$ then this condition even becomes vacuous, but there will be other $d$ where the $d$-core is self-conjugate.} Thus the results on unipotent character degree coincidences have impact on local-global conjectures in modular representation theory (see \cite{craven2012un,cravenrouquier2013} for more details).

\section{Maximal and average multiplicities}
\label{sec:mults}
\subsection{Symmetric groups}

Determining whether two partitions $\lambda$ and $\mu$ have the same product of hook lengths looks a difficult task, which is why in \cite{craven2008} the author worked with the tighter condition that the hook lengths are identical, not merely that their products were identical. This, of course, has the downside that one misses many pairs of partitions with the same product of hook lengths. In this section we look at both whether partitions have the same multiset of hook lengths and whether they have the same products of hook lengths. After this, we consider the situation for $\GL_n(q)$.

First we see what we can prove about the case where the multiset of hook lengths must be the same, rather than their product.

\begin{thm}\label{thm:boundsformultisets} \begin{enumerate}
\item There is a cluster of order $4$ and size $n$ if and only if $n\geq 22$ or $n=14,17,19,20$.
\item There is a cluster of order $8$ and size $n$ if $n\geq 206$. If $n\leq 83$, the largest order of a cluster of size $n$ is $4$, whereas there is a cluster of order $8$ and size $84$.
\end{enumerate}
\end{thm}
\begin{proof}
We first consider clusters of order $4$. In \cite{hermanchung1978}, Herman and Chung find two periodic clusters of order $2$:
\[\{(n+6,n+3,n+3,2),(n+5,n+5,n+2,1,1)\},\qquad  \{(n+8,n+4,n+3,3,1),(n+7,n+6,n+2,2,1,1)\}.\]
Together with their conjugates, this yields a cluster of size $4$ for all $3n+14$ and $3n+19$ with $n\geq 0$. In \cite{craven2008} another periodic cluster of period $3$ was found:
\[ \{(n+10,n+4,n+4,4,2),(n+8,n+8,n+2,2,2,1,1)\}.\]
This yields a cluster of size $4$ for all $3n+24$ with $n\geq 0$. This proves that there is a cluster of order $4$ and size $n$ for all $n\geq 22$ and $n=14,17,19,20$. To prove that there are no clusters of the remaining sizes we can simply do a computer search of all partitions of the remaining integers. This proves (i).

For (ii), we find nine periodic clusters of period $9$, order $8$ and various sizes, one for each congruence modulo $9$.
\begin{align*} \text{Size $180$:}&
\\&(28,21,20,19,17,13,12^2,11^2,4^3,3,1),\quad (26,25,18,17,16^2,15,11,9^2,8,2^4,1^2),
\\ &(25,23,19,18^2,17,16,9,8^2,6,2^6,1),\quad(23,22^2,21,17,15,14,13,6^2,5^2,4,1^7).
\\ \text{Size $190$:}
\\ &(29,23,20^2,18,15,13^2,11^2,5,4,3^2,1^2),\quad (28,25,19^2,17^2,15,12,10^2,7,3,2^3,1^2),
\\ &(26,24,21,19^2,17^2,11,8^2,6,3,2^4,1^3),\quad (25,23^2,21,18,16^2,13,7^2,5^2,3,2,1^6).
\\ \text{Size 128:}
\\ & (20,17^2,11,9^3,8^3,5^2,2),\quad (19^2,16,10^2,9^3,7^3,4,1^2),
\\ & (17,15^3,14^2,8,5^3,3^3,2^3),\quad (16^2,15^3,13,7^2,4^3,3^3,1^3).
\\ \text{Size 192:}
\\&(30,21^3,19,13^2,12^3,4^4,2),\quad(27^2,18^3,17^2,11,9^3,2^5,1^2),
\\&(27,25,19^2,18^3,9^3,7,2^7),\quad(24^2,23^2,17,15^3,6^3,5^2,1^8).
\\ \text{Size 130:}
\\ &(19,17,16^2,10,8^2,7^3,5,4^2,2),\quad (18^2,17,15,9^2,8^2,6^3,5,3,1^2),
\\ &(18,16^2,15^2,9,7,6^3,4^2,3^2,2),\quad (17^2,16^2,14,8^2,7,5^3,4^2,2,1^2).
\\ \text{Size 140:}
\\ &(21,17^3,12,9,8^4,4^3,3),\quad (20,17,16^3,11,7^4,4,3^4),
\\ &(19^3,15,10^3,9,6^4,2,1^3),\quad 
(18^3,17,14,9^3,5^4,4,1^4).
\\ \text{Size 159:}
\\ &(25,19^2,16,14,11^2,10^3,4^3,2),\quad(23^2,17,14^2,13^2,10,8^3,2^3,1^2),
\\&(22,20,17^2,16^2,13,7^3,5,2^6),\quad(20^2,19^2,16,14,11^2,5^3,4^2,1^6).
\\ \text{Size 214:}
\\ &(33,23^3,22,15,14^2,13^2,5,4^3,3,1),\quad (30,29,20^4,19,12,10^2,9,3,2^5,1^2),
\\ &(30,28,21,20^4,11,10^2,8,3,2^6,1),\quad(27,26^2,25,18,17^3,7^2,6^2,5,2,1^8).
\\ \text{Size 170:}
\\ &(25,21,20,18,16,13,11^2,9^2,5,4,3^2,1^2),\quad (24,22,21,17,15,13^2,11,8^2,6,5,2^2,1^3),
\\ &(24,21,19^2,17,15,11,10,8^2,5,3^3,2,1^2),\quad(23,21^2,19,16,14,12,11,7^2,5^2,3,2,1^4).
\end{align*}

Since this covers all nine congruence classes, we prove the result. Using the \texttt{symmetric} Magma package, all clusters of size $n$ have been constructed for $n\leq 84$, and they all have order at most $4$ apart from one for $n=84$.
\end{proof}

There is an obvious question here, which is whether there are clusters of size $8$ in the range $85\leq n\leq 205$. The answer is `maybe'. First, by working with periodic clusters of order $8$ and different periods, from $7$ to $10$, one is able to find clusters of order $8$ and all sizes between $182$ and $205$, so one can reduce the bound to $182$. (All such clusters appear in the periodic cluster database in the Magma package \texttt{symmetric}.)

The smallest cluster of order $8$ has size $84$ (it is of period $7$), and there are known (periodic) clusters of order $8$ of all sizes from $84$ to $205$ apart from
\[ \{ 85, 86, 87, 88, 89, 90, 92, 93, 94, 95, 96, 97, 99, 100, 101, 102, 103, 104,
106, 107, 108, 109,\]
\[ 110,111, 114, 115, 116, 117, 118, 122, 125, 132, 135, 136,
138, 142, 143, 150, 151, 181 \}. \]
It is not known whether there are clusters of order $8$ of these sizes, but it is certainly likely that there are clusters of some of the larger sizes in the list above. There is \emph{no} cluster of order $8$ and size up to $83$ by an exhaustive computer search, but computer programs take increasingly inconveniently long times to prove this.\footnote{The case $n=82$ took around two weeks using a Magma program, just because of the sheer number of partitions involved. By $n=90$, the largest case known, it takes months, with no real option for parallelization.} From an exhaustive search, there is a single cluster of order $8$ of sizes $84$ and $85$ (all other clusters have order $4$ or smaller), and no cluster of order greater than $4$ of sizes $86$ to $90$.

As an example we note that, of the 239 clusters of size $70$ and order $4$, only 89 come from a periodic cluster of size $2$ (together with conjugate partitions). Accessing non-periodic clusters without an exhaustive search looks difficult.

\medskip

Moving from the multisets having to coincide to coincidences of character degrees, we can prove much better bounds than those in Theorem \ref{thm:boundsformultisets}. First, for $n\leq 129$ we can use Magma to explicitly calculate $m(n)$.\footnote{The reason we stop at $n=129$ is technical. This is the point at which the author's programs start to break down, because we start reaching what appear to be limits on the maximum size of a set in Magma, or something like the maximum number of bytes available to store a set. It was not processor or memory limits. One could go slightly higher by splitting the set of character degrees up and repeatedly running the program, but such manoeuvres will only eke out a few more cases, with one having to slice the set of character degrees up in ever finer ways. One might need to go to, perhaps, $n=200$ to see a better pattern in Figure \ref{fig:avmult}, and that is simply infeasible.} It is given in Table \ref{t:maxmult}. Combining this with Theorem \ref{thm:boundsformultisets} proves almost all of the next result.

\begin{thm} \label{thm:boundsfordegrees}
\begin{enumerate}
\item We have that $m(n)\geq 4$ if and only if $n\geq 11$ or if $n=6,7$.
\item We have that $m(n)\geq 6$ if and only if $n\geq 17$ or if $n=13,14,15$. 
\item We have that $m(n)\geq 8$ if and only if $n\geq 21$ or if $n=17,19$.
\end{enumerate}
In all cases, the same statement holds when restricting ones attention only to characters parametrized by partitions that are not self-conjugate. In particular, this means that for $n\geq 17$, $A_n$ possesses three irreducible characters of the same degree that extend to $S_n$.
\end{thm}
\begin{proof} For the bound for $m(n)\geq 4$, we can use Theorem \ref{thm:boundsformultisets} to prove the result for $n\geq 22$, and then Table \ref{t:maxmult} for the remaining values of $n$. For the second and third parts, Theorem \ref{thm:boundsformultisets} proves the result for $n\geq 206$, and from the remark afterwards, we know the result for $n\geq 182$. Table \ref{t:maxmult} gives us the result for $n\leq 129$. At this point one can either check each value in the range $130\leq n\leq 181$ individually that $m(n)\geq 8$ just by calculating all character degrees one at a time until at least eight of the same are found, or use the other known clusters of order $8$ to reduce the number of possible $n$ for which a brute-force search is needed. Either way, this proves the result.

To prove the tighter statement about the partitions being not self-conjugate, periodic clusters automatically satisfy this statement, so we merely have to prove the result in the range $2\leq n\leq 181$. In the \texttt{symmetric} package, the program \texttt{MaximumMultiplicityOfSymmetricCharacterDegreeExceeds}, which is fairly quick, has a parameter \texttt{IgnoreSelfConjugatePartitions}. Setting this to be true ignores irreducible characters parametrized by self-conjugate parameters, and can be used to verify the result in this range.
\end{proof}

From Table \ref{t:maxmult} we see what the answer should be for larger bounds: for $m(n)\geq 10$ we need $n\geq 25$, for $m(n)\geq 12$ we need $n\geq 28$, and so on. However, this is conjectural because we do not have a version of Theorem \ref{thm:boundsformultisets} for these values, to eliminate all but a small number of $n$ that we can check individually. The proof in \cite{craven2008} gives no explicit bounds, and the version in \cite{craven2010a} gives us explicit bounds of the form $m(n)\geq O(n^{1/7})$. These are obviously not useful for larger values of $m(n)$.

\medskip
One other option is, instead of considering the largest multiplicity $m(n)$, to consider the \emph{average} multiplicity. This is simply $k(G)/|\cd(G)|$. For symmetric groups, $k(S_n)$ is obviously the number of partitions of $n$, so we are dividing the number of partitions of $n$ by the number of irreducible character degrees of $n$. If there were no coincidences of character degrees apart from conjugate partitions, the average multiplicity should be slightly below $2$, because the number of self-conjugate partitions is small compared with the number of partitions.

However, we know that there are coincidences of character degrees, since $m(n)\to\infty$ as $n\to\infty$. Figure \ref{fig:avmult} gives the ratio $k(S_n)/|\cd(S_n)|$ for $1\leq n\leq 129$. It looks far from conclusive at this stage whether the average multiplicity has a limit (maybe 2.5?\footnote{A suggestion from one of the referees is that it could possibly be related to the constant $\pi\sqrt{2/3}\sim 2.565$, since this quantity appears in the asymptotic formula for the number of partitions, \[p(n)\sim \frac{e^{\pi\sqrt{2n/3}}}{4n\sqrt3}.\]}) or tends to infinity as $n$ increases.\footnote{In {{\cite{moreto2007}}}, Moret\'o calculated this ratio for $n\leq 55$, but the general trend is not particularly obvious for those smaller values of $n$.}

\ThisAltText{Graph of the number of conjugacy classes divided by the number of distinct character degrees for n up to 129. The graph is quite jagged, but increases overall but more slowly as n increases. For n above 100 the graph is slightly jagged and very slowly increasing.}
\begin{figure}\centering \includegraphics[width=0.9\textwidth]{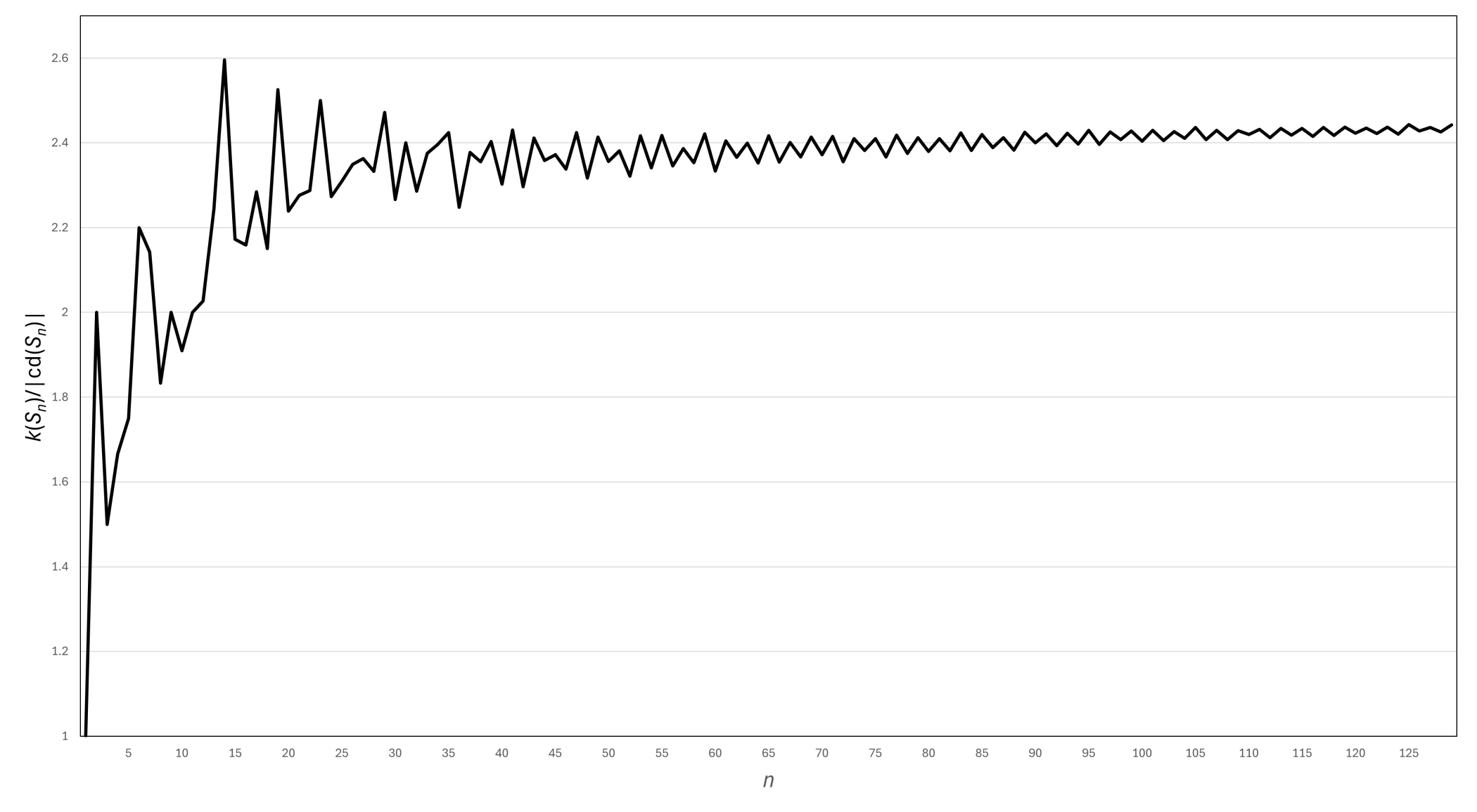}\caption{Graph of the ratio $k(S_n)/|\cd(S_n)|$ for $1\leq n\leq 129$.}\label{fig:avmult}
\end{figure}

We can ask the question for arbitrary finite groups, a significant strengthening of the result by Moret\'o \cite{moreto2007} and the author \cite{craven2008} on $|G|$ being bounded by a function of $m(G)$. In \cite[p.245]{moreto2007}, Moret\'o asks this question specifically for soluble groups rather than all groups, but it looks as though this might be true, at least for all groups other than symmetric groups.

\begin{conj}\label{conj:avmult} There exists a function $f:\Q\to\Q$ such that, if $G$ is a finite group with no alternating composition factor, then $|G|\leq f(k(G)/|\cd(G)|)$.
\end{conj}

The only reason that alternating groups are excluded from this is that Figure \ref{fig:avmult} does not compel the author to believe that the average multiplicity is unbounded for symmetric groups. We discuss for which groups this conjecture is known below, after the next conjecture.

One could ask the dual question, and consider the multiset of conjugacy class sizes, and their multiplicities. This was considered in \cite{jaikin2005}, where it was proved that the orders of finite soluble groups are bounded by a function of the maximal multiplicity of their conjugacy class sizes, and partially extended to all finite groups in \cite{nguyen2011}. In analogy with $\cd(G)$, write $\cc(G)$ for the set of conjugacy class sizes.

\begin{conj}\label{conj:avmultcc} There exists a function $f:\Q\to\Q$ such that, if $G$ is a finite group then $|G|\leq f(k(G)/|\cc(G)|)$.
\end{conj}

In other words, the average multiplicity of a conjugacy class size should tend to infinity as the group order does. The evidence for this for symmetric groups is much more convincing, and looks potentially easier to prove since centralizer sizes are easier to compute than character degrees. The graph, Figure \ref{fig:avmultcc}, is a smooth exponential-like curve as well, suggesting an asymptotically accurate formula might even be findable.\footnote{Although it looks like an exponential curve, it cannot of course be strictly exponential because $k(S_n)$ is not exponential in $n$, it asymptotically approaching $e^{a\sqrt n}/bn$, as mentioned in a previous footnote.}
\ThisAltText{Graph of the number of conjugacy classes divided by the number of distinct conjugacy class sizes for n up to 85. The graph is very smooth and curves upwards in an superpolynomial-like way.}
\begin{figure}\centering \includegraphics[width=0.9\textwidth]{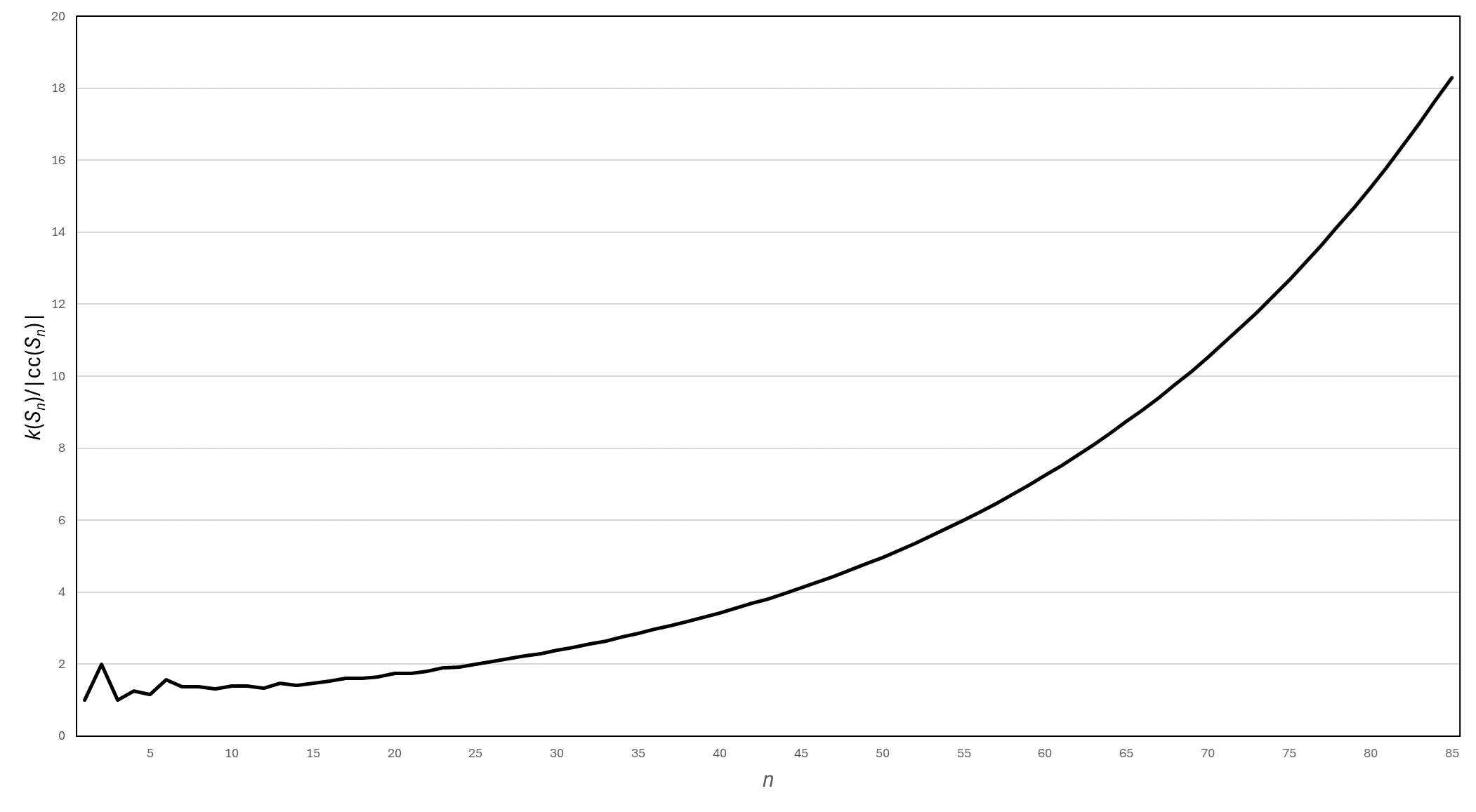}\caption{Graph of the ratio $k(S_n)/|\cc(S_n)|$ for $1\leq n\leq 85$.}\label{fig:avmultcc}
\end{figure}

We note that Conjectures \ref{conj:avmult} and \ref{conj:avmultcc} hold for $p$-groups and soluble groups with trivial Frattini subgroup \cite[Lemmas 2.2 and 2.3]{moreto2023}, and Lie-type groups of bounded rank by \cite[Lemma 6.3]{moreto2007}. Conjecture \ref{conj:avmult} is true for simple groups of Lie type by \cite[Corollary 1.6]{liebeckshalev2005}.

\medskip

Having discussed the first column of the character table, one could ask the question about more than one column. A question of Alexander Ryba was whether there are irreducible characters $\lambda$ and $\mu$ of $S_n$ for $n\geq 2$ such that $\lambda(1)=\mu(1)$ \emph{and} $\lambda(g)=\mu(g)\neq 0$ for any $g$ of cyclic type $2,1^{n-2}$ (so that $\lambda$ and $\mu$ are not conjugate)? The first such case is for $S_{12}$, where the characters associated to the partitions $(7,1,1,1,1,1)$ and $(4,4,4)$ have the same character degree and value on an involution. By $S_{30}$ there are such 70 pairs of partitions, and by $S_{40}$ there are 314. The first $n$ for which there are two characters with the same value on the conjugacy classes of $1$, $(1,2)$ and $(1,2,3,4)$ simultaneously is $35$.

In the other direction, there are four characters of $S_{32}$ that share the same non-zero values on $1$ and $(1,2)$ (so that they are not partitions and their conjugates), leading to the following reasonable-sounding conjecture.

\begin{conj} Let $X$ be a finite set of permutations and let $k\geq 1$ be an integer. For all sufficiently large $n$, there exist $k$ distinct irreducible characters of $S_n$ whose values on permutations $g\in S_n$ coincide for all $g$ of cycle type in $X$.
\end{conj}

This conjecture contrasts with a recent result of Miller \cite{miller2024un} which is that, in some sense, you don't need many character values of $\chi^\lambda$ to determine $\lambda$. Why this does not contradict this conjecture is because in Miller's algorithm you can choose the next class to examine the character value for based on the previous known character values, whereas in the conjecture you have to choose the classes in advance.

A result of Wildon \cite[Proposition 2.1]{wildon2008} states that if $\chi$ and $\psi$ are irreducible characters of $S_n$, and $\chi$ and $\psi$ agree on all $i$-cycles for $1\leq i\leq n$, then $\chi=\psi$. Thus we need only consider cycles in the previous conjecture, where the Murnaghan--Nakayama and Frobenius formulae are of a simpler form. Work around Kerov's character polynomials (see \cite[Section 1.2]{gouldenrattan2007}), which specifically consider character values on cycles, might be of use when attacking this problem.

There is also an alternative formulation of the previous conjecture that seems more attractive. That $|S_n|$ is bounded in terms of $m(n)$ is the case $m=1$ of this next conjecture. It is unclear to the author whether even the case $n=2$ follows from the construction for the case $m=1$ from \cite{craven2008}.

\begin{conj} Let $m,k\geq 1$. For all sufficiently large $n$, there are $k$ (distinct) irreducible characters of $S_n$ whose restrictions to $S_m$ coincide.
\end{conj}

\subsection{General linear groups}

If one considers generic degrees for unipotent characters of $\GL_n(q)$ rather than symmetric groups, the conjugate partition does not yield a character of the same degree in general, as Lusztig's $a$-function should differ between a partition and its conjugate. If $n=15$ then the partition $(6,3,2,2,2)$ has the property that it and its conjugate have the same $a$-function, which takes value $21$. This is the smallest such coincidence: there are two for $n=16$, none for $n=17,18,23$, and at least one for all larger numbers. At the moment there is no general theory for generic degree multiplicities, so one cannot say much about such partitions. We can at least prove that there are always a pair of partitions with the same generic degree for all $n\geq 24$.

\begin{prop}\label{prop:alwaystwogln} If $n=15,16,19,20,21,22$ or $n\geq 24$ then there is a partition $\lambda$ such that $a(\lambda)=a(\lambda')$, where $\lambda'$ is the conjugate partition of $\lambda$. Otherwise there is no such partition.
\end{prop}
\begin{proof}
We can prove the result generically for $n\geq 28$. The partition $\lambda=(7^2,4,3^3,1)$ has conjugate $\lambda'=(7,6^2,3,2^3)$. A direct check shows that $a(\lambda)=a(\lambda')=57$. We can add $m$ to the first part and $m$ parts of size $1$ to $\lambda$ to make the partition $\lambda^{(m)}$. We have that \[a(\lambda^{(m)})=57+(7+8+\cdots+(m+6)),\]
and this is clearly the same for the conjugate partition. Thus there is a partition of $28+2m$ with the desired property for all $m\geq 0$.

Moreover, there is an addable box to $\lambda$ in the $(4,4)$ position. This gives the partition $\mu=(7^2,4^2,3^2,1)$, with conjugate $\mu'=(7,6^2,4,2^3)$. As the box lies in the diagonal, $a(\mu)=a(\mu')=60$. We can form the partition $\mu^{(m)}$ in analogy with $\lambda^{(m)}$ by adding a box in the $(4,4)$ position to $\lambda^{(m)}$ (or by extending the first row and column of $\mu$) to obtain a partition with the desired property for all odd integers at least $29$.

For integers less than $28$, we simply list partitions with the desired property.
\[ (6,3,2^3), \quad (6,3^2,2^2),\quad (7,3^3,2,1),\quad (6,4^2,3^2),\quad (6,4^3,3),\quad (8,4,3,2^3,1),\]
\[ (8,4^2,3,2^2,1),\quad (8,4^3,2^2,1),\quad (9,4,3^3,2^2),\quad (8,5^2,3,2^3).\]

Proving that no other integer works requires checking all partitions for each individual $n$, which uses a computer (or can be done by hand for very small $n$).
\end{proof}

On the other hand, using the database of all clusters of partitions of order at least 4 and size at most 90 in the \texttt{symmetric} package, we find that only 74 of the 13814 such clusters of partitions have unipotent generic degrees that are not all distinct, so making the multiset of hook lengths \emph{and} the $a$-function coincide is much more restrictive a condition. The general methods and constructions from \cite{craven2008} give no control over the $a$-function, so this area is wide open. It seems ambitious to offer a conjecture on no evidence at all, so instead we should phrase it as a question.

\begin{ques} Does the maximal multiplicity of a generic degree among unipotent characters of $\GL_n(q)$ tend to infinity as $n$ tends to infinity? Does it even exceed $2$?
\end{ques}

It seems reasonable that the maximal multiplicity will, at some point, exceed $2$, although it does not for $n\leq 90$. (We can check this because all clusters of size at most $90$ are known, and in the \texttt{symmetric} package database.) It is not remotely clear at this stage whether the generic degree multiplicities should tend to infinity, or even be unbounded.

\section{Largest degrees for symmetric groups}
\label{sec:chardegs}
We now look at the largest degrees of symmetric groups. The first article about this was by McKay \cite{mckay1976}, who computed $b_1(n)$ for $n\leq 75$. In the \texttt{symmetric} package we list the values for $b_1(n)$, $b_2(n)$ and $b_3(n)$ for $n\leq 150$, together with the partitions that reach these values. This allows us to consider the fraction $b_i(n)/b_1(n)$ for $i=2,3$.

As stated in the introduction, Halasi, Hannusch and Nguyen \cite{hhn2016} proved that the sums of the squares of all irreducible character degrees other than the largest exceeds twice the square of the largest degree.

Figure \ref{fig:b3b1} gives the ratio $b_3(n)/b_1(n)$ for $n$ up to $150$, and from this it is easy to see that $b_3(n)/b_1(n)$ and $b_2(n)/b_1(n)$ are bounded below by $0.8$ for all $52\leq n\leq 150$, and highly likely that this is true for all larger $n$. If $b_3(n)$ and $b_2(n)$ are each achieved by a unique (necessarily self-conjugate) character then the sum of the squares of character degrees that are $b_2(n)$ or $b_3(n)$ is less than $2b_1(n)^2$, and otherwise it exceeds $2b_1(n)^2$. (See Conjecture \ref{conj:chardegssquares} below.)
\ThisAltText{Graph of the ratio of the third largest character degree over the largest character degree for n up to 150.\ Very jagged but steadily increasing until all values are above 0.85 for large n.}
\begin{figure}[htbp] \centering \includegraphics[width=0.9\textwidth]{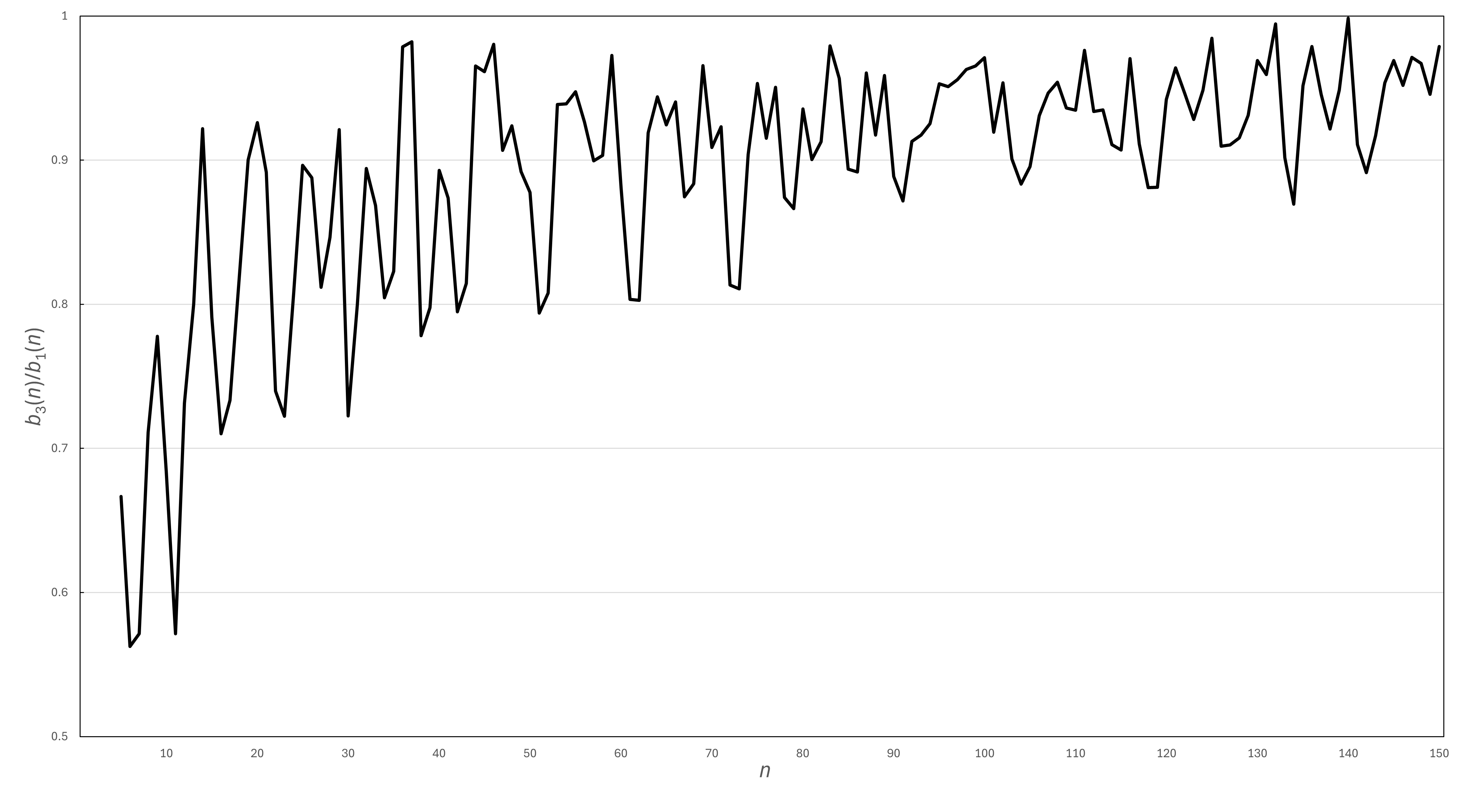}\caption{Graph of the ratio $b_3(n)/b_1(n)$ for $5\leq n\leq 150$.}\label{fig:b3b1}
\end{figure}

Can both the second- and third-largest degrees be achieved by unique characters? Yes. The cases $n=19,59$ are the only ones for which this occurs up to $n=150$, but there is no reason to believe that it will not occur again. Thus we have the following conjecture, where we have to exclude some small $n$.

\begin{conj}\label{conj:chardegssquares} Let $n$ be an integer with $n\geq 5$, $n\neq 6,7$.
\begin{enumerate}
\item If $n\neq 11$, then $b_1(n)^2\leq b_2(n)^2+b_3(n)^2$.
\item If $n\neq 6,7,11,16,38$, and $b_2(n)$ or $b_3(n)$ is achieved by two or more irreducible characters of $S_n$, then
\[ \sum_{\substack{\chi\in\Irr(S_n)\\\chi(1)\in \{b_2(n),b_3(n)\}}}\chi(1)^2>2b_1(n)^2.\]
\item If $b_2(n)$ and $b_3(n)$ are both achieved by a unique irreducible character of $S_n$ then
\[ \sum_{\substack{\chi\in\Irr(S_n)\\\chi(1)\in \{b_2(n),b_3(n),b_4(n)\}}}\chi(1)^2>2b_1(n)^2.\]
\item If $n\geq 62$ then $2b_1(n)^2\leq b_2(n)^2+b_3(n)^2+b_4(n)^2$.
\end{enumerate}
\end{conj}
For the last part of this conjecture, in the region $13 \leq n\leq 61$ the statement holds unless $n$ lies in
\[ \{15,16,17,22,23,27,30,34,38,42,51,61\}.\]
This conjecture has been checked up to $n=100$ (the value for $b_4(n)$ is not in the \texttt{symmetric} package, and had to be computed separately), and should hold for large $n$ as the ratios $b_3(n)/b_1(n)$ and $b_2(n)/b_1(n)$ approach $1$, so the inequality almost holds even without the contribution of $b_4(n)^2$.

\medskip

There have been a number of papers in the last few decades about large-degree characters, and a continuous version of the hook formula. We specifically mention \cite[Theorem 2]{vershikkerov1985}; this theorem states, loosely speaking, that the squares of the degrees of the `large-degree' characters of $S_n$ sum to $n!-o(n!)$, i.e., most of the order of the group. On the other hand, \cite[Theorem 1]{vershikkerov1985} bounds $b_1(n)$ well away from $\sqrt{n!}$, meaning that there must be unboundedly many large-degree characters of $S_n$.

The main barrier to these two theorems implying the following conjecture is that we do not know whether there are only boundedly characters for each of these large degrees $b_i(n)$.

\begin{conj}\label{conj:ratiobmtob1} For any $m\geq 2$, the ratio $b_m(n)/b_1(n)\to 1$ as $n\to\infty$.
\end{conj}

Although there is a significant amount of noise in Figure \ref{fig:b3b1}, there is a clear trend. One can see this intuitively: the expectation is that, for large $n$, partitions that are close to the LSVK curve (which was independently discovered by Logan and Shepp \cite{loganshepp1977} and Vershik and Kerov \cite{vershikkerov1977}) should yield the largest-degree characters. For large $n$, there are many partitions that are arbitrarily close to this curve, so for any fixed $m$ there are $m$ characters of degree all very close to the maximum.

Despite several papers being written on the subject (most recently, for example, \cite{duzhinsmirnovmalsev2023}), we do not have an algorithm to give the largest degree, even conjecturally. Ideas like that the largest character degree should be parametrized by a partition $\lambda$ that is `very close' to a self-conjugate partition (as considered in \cite[Section 3]{vershikpavlov2010} and subsequent papers by this group of authors) appear perhaps to not hold in general, based on our evidence here. For example, if $n=96,109,123,138,148$ then one must remove four boxes to arrive at a self-conjugate partition. (These are the only such $n$ up to $150$.) This is despite the fact that $b_1(n)$ for $n=94,107,121,136$ being realized by a self-conjugate partition; in general, one cannot find $b_1(n)$ by adding a box to the partition realizing $b_1(n-1)$ and choosing the partition with largest degree, i.e., inducing the largest-degree character from $S_{n-1}$. (This is not a new observation.) The first $n$ for which one needs to add three boxes to a self-conjugate partition to obtain a partition realizing $b_1(n)$ is $n=37$, and then one needs three boxes for $n=45,54,83,95,97,102,108,110,122,124,137,139$ up to $150$.

Indeed, one actually needs to subtract five boxes from the partition witnessing $b_1(132)$ to obtain a self-conjugate partition, and then again for $n=147,149$. This leads the author to suspect that the answer to the following question is `no'.

\begin{ques}
Is there an integer $d\geq 5$ such that, if $\lambda$ is a partition of $n$ with $\chi_\lambda(1)=b_1(n)$, then $\lambda$ can be obtained from a self-conjugate partition by adding at most $d$ boxes?
\end{ques}

The reason that this question is relevant is that the methods and algorithms to generate large-degree characters, for example given in \cite{duzhinsmirnovmalsev2023}, are dependent on starting with something close to $b_1(n-1)$ and obtaining something close to $b_1(n)$ from it by adding a box. The behaviour of the function that sends $n$ to the number of boxes needed to be added to a self-conjugate partition to obtain partition with character degree $b_1(n)$ is that it is not smooth, and frequently jumps around, by ever-increasing amounts. This suggests that any form of `shaking algorithm'\footnote{Roughly speaking, a shaking algorithm involves finding a large-degree character of $S_n$, restricting it to $S_{n-k}$ for some fixed $k$, then inducing the factors of this back up to $S_n$ and looking for one with a larger degree. Combinatorially -- using the branching rule -- it involves finding all ways to remove $k$ boxes, add $k$ boxes back again, then determine the hook lengths of the resulting partition, hoping for something larger.} alluded to above (see, for example, \cite{duzhinsmirnovmalsev2023}) is doomed to fail to determine $b_1(n)$ efficiently in the long run.

On the other hand, $b_1(n)$ is achieved by a self-conjugate partition quite often -- for larger $n$, between $120$ and $150$ it is achieved for $n=121,125,127,129,134,136,142,144$. It is highly likely that $b_1(n)$ is realized by a self-conjugate partition infinitely often, but the author is not aware of any progress on that question. Finding a pattern to the $n$ here seems non-obvious.\footnote{I would like to thank one of the referees for suggesting these questions in this section.}

\begin{table}
\begin{center}\begin{tabular}{cc|cc|cc|cc}
\hline $n$ & $m(n)$ & $n$ & $m(n)$ &$n$ & $m(n)$ &$n$ & $m(n)$
\\ \hline
$2$&$2$&$34$&$18$&$66$&$34$&$98$&$104$
\\$3$&$2$&$35$&$30$&$67$&$54$&$99$&$96$
\\$4$&$2$&$36$&$14$&$68$&$38$&$100$&$104$
\\$5$&$2$&$37$&$20$&$69$&$56$&$101$&$110$
\\$6$&$4$&$38$&$26$&$70$&$50$&$102$&$106$
\\$7$&$4$&$39$&$16$&$71$&$48$&$103$&$112$
\\$8$&$2$&$40$&$20$&$72$&$44$&$104$&$102$
\\$9$&$3$&$41$&$22$&$73$&$50$&$105$&$146$
\\$10$&$2$&$42$&$20$&$74$&$44$&$106$&$104$
\\$11$&$4$&$43$&$26$&$75$&$58$&$107$&$120$
\\$12$&$4$&$44$&$25$&$76$&$46$&$108$&$114$
\\$13$&$6$&$45$&$24$&$77$&$60$&$109$&$132$
\\$14$&$6$&$46$&$24$&$78$&$48$&$110$&$126$
\\$15$&$6$&$47$&$32$&$79$&$58$&$111$&$136$
\\$16$&$4$&$48$&$16$&$80$&$64$&$112$&$126$
\\$17$&$8$&$49$&$32$&$81$&$72$&$113$&$130$
\\$18$&$6$&$50$&$30$&$82$&$56$&$114$&$108$
\\$19$&$10$&$51$&$26$&$83$&$58$&$115$&$144$
\\$20$&$6$&$52$&$24$&$84$&$66$&$116$&$120$
\\$21$&$8$&$53$&$32$&$85$&$86$&$117$&$172$
\\$22$&$8$&$54$&$32$&$86$&$68$&$118$&$114$
\\$23$&$12$&$55$&$40$&$87$&$86$&$119$&$142$
\\$24$&$8$&$56$&$32$&$88$&$78$&$120$&$158$
\\$25$&$12$&$57$&$34$&$89$&$94$&$121$&$160$
\\$26$&$12$&$58$&$32$&$90$&$102$&$122$&$140$
\\$27$&$10$&$59$&$32$&$91$&$84$&$123$&$176$
\\$28$&$12$&$60$&$34$&$92$&$84$&$124$&$178$
\\$29$&$22$&$61$&$44$&$93$&$94$&$125$&$188$
\\$30$&$14$&$62$&$30$&$94$&$80$&$126$&$194$
\\$31$&$12$&$63$&$44$&$95$&$104$&$127$&$164$
\\$32$&$12$&$64$&$36$&$96$&$96$&$128$&$174$
\\$33$&$16$&$65$&$52$&$97$&$104$&$129$&$224$

\\ \hline
\end{tabular}\end{center}
\caption{The function $m(n)$ for $2\leq n\leq 129$.}
\label{t:maxmult}
\end{table}

\begin{table}
\begin{center}\begin{tabular}{cc|cc|cc}
\hline $n$ & Number of pairs & $n$ & Number of pairs & $n$ & Number of pairs
\\ \hline
$15$&$1$&$51$&$311$&$87$&$23895$
\\$16$&$2$&$52$&$412$&$88$&$32343$
\\$17$&$0$&$53$&$380$&$89$&$31065$
\\$18$&$0$&$54$&$373$&$90$&$31942$
\\$19$&$1$&$55$&$500$&$91$&$37794$
\\$20$&$1$&$56$&$711$&$92$&$45633$
\\$21$&$5$&$57$&$786$&$93$&$52072$
\\$22$&$1$&$58$&$605$&$94$&$46924$
\\$23$&$0$&$59$&$695$&$95$&$54602$
\\$24$&$7$&$60$&$1331$&$96$&$78606$
\\$25$&$8$&$61$&$1241$&$97$&$77538$
\\$26$&$1$&$62$&$831$&$98$&$69055$
\\$27$&$5$&$63$&$1477$&$99$&$90644$
\\$28$&$13$&$64$&$2106$&$100$&$114893$
\\$29$&$6$&$65$&$1970$&$101$&$113891$
\\$30$&$10$&$66$&$1900$&$102$&$113139$
\\$31$&$10$&$67$&$2221$&$103$&$132908$
\\$32$&$17$&$68$&$2947$&$104$&$164782$
\\$33$&$30$&$69$&$3470$&$105$&$187349$
\\$34$&$11$&$70$&$3096$&$106$&$171432$
\\$35$&$25$&$71$&$3422$&$107$&$189781$
\\$36$&$57$&$72$&$5340$&$108$&$263020$
\\$37$&$33$&$73$&$5478$&$109$&$267378$
\\$38$&$21$&$74$&$4219$&$110$&$248100$
\\$39$&$49$&$75$&$6409$&$111$&$309322$
\\$40$&$93$&$76$&$8350$&$112$&$385150$
\\$41$&$64$&$77$&$7771$&$113$&$391224$
\\$42$&$60$&$78$&$8181$&$114$&$389427$
\\$43$&$93$&$79$&$9364$&$115$&$458881$
\\$44$&$108$&$80$&$12846$&$116$&$551870$
\\$45$&$172$&$81$&$14577$&$117$&$601380$
\\$46$&$111$&$82$&$11813$&$118$&$575444$
\\$47$&$102$&$83$&$14229$&$119$&$647147$
\\$48$&$293$&$84$&$21258$&$120$&$867144$
\\$49$&$284$&$85$&$21799$&$121$&$886962$
\\$50$&$172$&$86$&$18172$&$122$&$816004$
\\ \hline
\end{tabular}\end{center}
\caption{Number of pairs of partitions $\{\lambda,\lambda'\}$ of size $n$ with $a(\lambda)=a(\lambda')$, where $\lambda'$ is the conjugate of $\lambda$ (and $\lambda\neq\lambda'$).}\label{t:glnpartitionssamedegree}
\end{table}
\providecommand{\bysame}{\leavevmode\hbox to3em{\hrulefill}\thinspace}
\providecommand{\MR}{\relax\ifhmode\unskip\space\fi MR }
\providecommand{\MRhref}[2]{%
  \href{http://www.ams.org/mathscinet-getitem?mr=#1}{#2}
}
\providecommand{\href}[2]{#2}

\end{document}